\documentclass[a4paper,11pt]{article}

\usepackage{amsmath,amssymb,amsthm}
\usepackage[dvips]{graphicx,color}

\numberwithin{equation}{section}

\newtheorem{thm}{Theorem}[section]

\newtheorem{lem}[thm]{Lemma}
\newtheorem{cor}[thm]{Corollary}
\newtheorem{rem}[thm]{Remark}
\newtheorem{bthm}{Theorem}
\newtheorem{athm}{Theorem}

\theoremstyle{definition}

\newcommand{\rn}{\Bbb R^N_+}
\newcommand{\e}{\varepsilon}

\allowdisplaybreaks

\begin{document}
\title{A Nonlinear Elliptic PDE with Two Sobolev-Hardy\\Critical Exponents}

\author{YanYan Li\thanks{Partially
 supported by
        NSF grant DMS-0701545.} and Chang-Shou Lin}

\date{}

\maketitle

\abstract{
In this paper, we consider the following PDE involving  two Sobolev-Hardy critical exponents,

\begin{equation}
\label{0.1}
\left\{
\begin{aligned}
& \Delta u + \lambda\frac{u^{2^*(s_1)-1}}{|x|^{s_1}} + \frac{u^{2^*(s_2)-1}}{|x|^{s_2}} =0 \,\, \text{in } \Omega, \\
& u=0 \qquad  \text{ on } \Omega,
\end{aligned}
\right.
\end{equation}
where $0 \le s_2 < s_1 \le 2$, $0 \ne \lambda \in \Bbb R$ and $0 \in \partial \Omega$. 
The existence (or nonexistence) for least-energy solutions has been extensively studied when $s_1=0$ or $s_2=0$. In this paper, 
we prove that
 if
$0< s_2 < s_1 <2$ and the mean curvature
of $\partial \Omega$ at $0$
 $H(0)<0$, then \eqref{0.1} has a least-energy solution.  Therefore, this paper has completed the study of \eqref{0.1} for the least-energy solutions. We also prove existence or nonexistence of positive entire solutions of \eqref{0.1} with $\Omega =\rn$
under different situations of  $s_1, s_2$ and $\lambda$. 

}

\section{Introduction}
\hspace{.6cm}
Let $0\leq s\leq 2$, $2^*(s)=\frac{2(N-s)}{N-2}$ and $L^{2^*(s)}(\frac{dx}{|x|^s})$ denote the space of $f$ with $\int|f|^{2^*(s)}\frac{dx}{|x|^s}<+\infty$.
It is well known that the inclusion $H^1_0(\Omega)\hookrightarrow L^{2^*(s)}(\frac{dx}{|x|^s})$ is a family of non-compact embeddings.
In this paper, we want to study the combined effect of two such Sobolev-Hardy critical exponents on a nonlinear partial differential equation.
More precisely, we consider
\begin{equation}\label{1equation}
\begin{cases}
&\Delta u+\lambda\frac{u^{2^*(s_1)-1}}{|x|^{s_1}}+\frac{u^{2^*(s_2)-1}}{|x|^{s_2}}=0\quad\text{in \,}\Omega,\\
&u(x)>0\quad\text{in \,}\Omega\quad u(x)=0\quad\text{on \,}\partial\Omega,
\end{cases}
\end{equation}
where $0\leq s_2<s_1\leq 2$ and $\lambda\in\Bbb R$. Throughout the paper, $\Omega$ is a bounded smooth domain in $\Bbb R^N$ with $0\in\partial\Omega$.

\medskip

Our motivation for studying equation\,\eqref{1equation} comes from the celebrated  Caffarelli-Kohn-Nirenberg inequality\,\cite{CKN}:
there exists a constant $C$ such that for any $u\in C_0^\infty(\Bbb R^N)$, the inequality
\begin{equation*}
\int_{\Bbb R^N}|x|^{-bq}u^qdx\leq C\int_{\Bbb R^N}|x|^{-2a}|\nabla u|^2dx
\end{equation*}
holds, where $-\infty<a<\frac{N-2}{2}$, $0\leq b-a\leq 1$ and $q=\frac{2N}{N-2+2(b-a)}$.
Let $D_a^{1,2}(\Omega)$ be the completion of $C_0^\infty(\Omega)$ with the norm $\|u\|_a^2=\int_{\Bbb R^N}|x|^{-2a}|\nabla u|^2dx$, and set
\begin{equation*}
S(a,b;\Omega)=\inf_{u\in D^{1,2}_a(\Omega)\setminus\{0\}}\frac{\int_\Omega |x|^{-2a}|\nabla u|^2dx}{\left(\int_\Omega |x|^{-bq}|u|^q dx\right)^{\frac{2}{q}}}.
\end{equation*}
Naturally, we ask whether the best constant $S(a,b;\Omega)$ can be attained by some $u\in D^{1,2}_a(\Omega)\setminus\{0\}$.
For the past twenty years, this problem has been extensively studied. For recent development, we refer the readers to \cite{BPZ,CH,CW1,CLW,CC,TE,GK,GR,L,LIN,LW,HL} and the references therein.

When $0\in\partial\Omega$, this problem was first studied by Ghoussoub-Kang\,\cite{GK} and Ghoussoub-Robert\,\cite{GR}, also see \cite{CLW}.
In \cite{CLW}, among other things, Chern and the second author of this paper  proved the following theorem.

\begin{athm}\label{thmathm}
Suppose $0\in\partial\Omega$ and the mean curvature
 $H(0)<0$.
Then the best constant $S(a,b;\Omega)$ can be achieved in $D^{1,2}_a(\Omega)$ if $a, b, q$ satisfy one of the following conditions:

\medskip

\noindent
(i) $a<b<a+1$ and $N\geq 3$,

\medskip

\noindent
(ii) $b=a>0$ and $N\geq 4$.
\end{athm}

When $a=0$ and $0<b<1$, Theorem\,\ref{thmathm} was first proved by Ghoussoub and Robert\,\cite{GR}.
The proof of Theorem\,\ref{thmathm} in \cite{CLW} is to make use of a transformation: $u(x)=|x|^{-a}v(x)$.
Straightforward computations give
\begin{equation*}
\int_\Omega |x|^{-2a}|\nabla v|^2 dx=\int_\Omega |\nabla u|^2dx-\lambda\int_\Omega\frac{u^2}{|x|^2}dx,
\end{equation*}
where $\lambda=a(N-2-a)$. Then $S(a,b;\Omega)$ is equal to the following best constant:
\begin{equation}\label{2equation}
S_\lambda(\Omega)=\inf_{u\in H^1_0(\Omega)\setminus\{0\}}
\frac{\int_\Omega |\nabla u|^2dx-\lambda\int_\Omega\frac{|u|^2}{|x|^2}dx}{\left(\int_\Omega\frac{|u|^{2^*(s)}}{|x|^s}dx\right)^{\frac{2}{2^*(s)}}},
\end{equation}
where $\lambda=a(N-2-a)$ and $s=(b-a)q\in[0, 2)$ if $b<a+1$. Note that if $b=a+1$, thus $s=2$ and the question for the best constant is a linear problem. Hence,
we always exclude the case $b=a+1$. By \eqref{2equation}, Theorem\,\ref{thmathm} is equivalent to saying that
equation\,\eqref{1equation} has a solution provided that either (i) $N\geq 3$, $\lambda<(\frac{N-2}{2})^2$, $0<s_2<s_1=2$, or
(ii) $N\geq 4$, $0<\lambda<(\frac{N-2}{2})^2$, $s_1=2$ and $s_2=0$.

To study equation \eqref{1equation}, we consider the nonlinear functional $\varPhi$:
\begin{equation*}
\varPhi(u)=\frac{1}{2}\int_\Omega|\nabla u|^2 dx-\frac{\lambda}{p_1+1}\int_\Omega\frac{(u^+)^{p_1+1}}{|x|^{s_1}}dx-\frac{1}{p_1+1}\int_\Omega\frac{(u^+)^{p_2+1}}{|x|^{s_2}}dx
\end{equation*}
for $u\in H^1_0(\Omega)$, where for the simplicity of notations, we let $p_1=2^*(s_1)-1$ and $p_2=2^*(s_2)-1$.
It is easy to see that there is positive constants $\rho_0, c_0>0$ such that
\begin{equation*}
\varPhi(u)\geq c_0\text{ \,if \,}\|u\|_{H^1_0}=\rho_0.
\end{equation*}
Note that $p_2 > p_1$ because $s_1 > s_2$. Thus, no matter what 
the sign of $\lambda$ is, there is $u_0 \in H^1_0(\Omega)$ such that 
$\Phi(u_0) \le 0$. 
Set
\begin{equation}\label{4equation}
c_*=\inf_{P\in{\cal P}}\max_{w\in P}\varPhi(w),
\end{equation}
where ${\cal P}$ is the class of continuous paths in $H^1_0(\Omega)$ connecting $0$ and $u_0$.
We note that since $p_2>p_1$, the function $t\to\varPhi(tu)$ has the unique maximum for $t\geq 0$.
Furthermore, we have
 $$c_*=\displaystyle\inf_{u\in H^1_0(\Omega)}\max_{u\geq 0, \,u\not\equiv 0}\varPhi(tu).$$
It is well-known that due to the non-compact embedding of $H^1_0\hookrightarrow L^{2^*(s)}(\frac{dx}{|x|^s})$,
$\varPhi$ does not satisfy the Palais-Smale condition. Therefore, in general $c_*$ might not be a critical value for $\varPhi$.
As usual, if $c_*$ is a critical value, and $u$ is a critical point of $\varPhi$ with $\varPhi(u)=c_*$,
then $u$ is called a least-energy solution.

\medskip

 When
 $s_2=0$, equation\,\eqref{1equation} becomes
\begin{equation}\label{3equation}
\Delta u+\lambda\frac{u^{2^*(s_1)-1}}{|x|^{s_1}}+u^{\frac{N+2}{N-2}}=0\quad\text{in \,}\Omega.
\end{equation}
\noindent
When $\lambda<0$, the best constant $S_\lambda(\Omega)$ of \eqref{2equation} always satisfies
\begin{equation*}
S_\lambda(\Omega)=S_0(\Omega)=S_N,
\end{equation*}
where $S_N$ is the Sobolev best constant. 
Thus, $S_{\lambda}(\Omega)$ can not be attained in $H^1_0(\Omega)$, and as a consequence, 
  $c_*$ could not be a critical value of $\varPhi$.
 In fact, for $0\leq s_1<2$, it is not difficult to see that the constant $c_*$ of \eqref{4equation}
is always equal to $\frac{1}{N}S_N^{\frac{N}{2}}$ and $c_*$ is never a critical value for $\varPhi$. 
Thus, there exist no least-energy solutions for equation\,\eqref{3equation} when $\lambda<0$. However,  when $\lambda>0$, $0<s_1<2$ and $s_2=0$, the following theorem was proved in \cite{HLW}.

\begin{bthm}
 Suppose $N\geq 4$ and $0\in\partial\Omega$ with $H(0)<0$.
Then equation\,\eqref{3equation} has a solution, provided that $\lambda>0$, and $0 < s_1 \le 2$.
\end{bthm}

In summary,  equation\,\eqref{1equation} has been studied for either $s_1=2$
or $s_2=0$. The purpose of this paper is to study the remaining cases for equation\,\eqref{1equation}. The following is one of our main theorem.
\begin{thm}\label{1theorem}
Suppose $\Omega$ is a bounded smooth domain in $\Bbb R^N$, $0\in\partial\Omega$ and the mean curvature $H(0)<0$.
Then equation\,\eqref{1equation} has a least-energy solution if
\begin{equation*}
N\geq 3, \,\lambda\in \Bbb R\text{ \,and \,}0<s_2<s_1<2.
\end{equation*}
\end{thm}

\medskip

In principle, the solvability of least energy solutions is closely related to 
 the existence of the entire solutions of equation\,\eqref{1equation}, i.e., $\Omega=\Bbb R^N_+$, the upper half-space.
The existence of entire solutions on the upper half space has been proved by Bartsch, Peng and Zhang\,\cite{BPZ}
when $0<s_2<s_1=2$ and $\lambda<(\frac{N-2}{2})^2$, by Musina \cite{MU} when $N \ge 4$, $s_2=0$, $s_1=2$ and $0<\lambda<(\frac{N-2}{2})^2$,
and by Hsia, Lin and Wadade\,\cite{HLW} when $s_2=0$, $0<s_1<2$ and $\lambda>0$. Close to Theorem~\ref{1theorem}, the following existence of positive entire solutions will be proved in this paper.

\begin{thm}\label{2theorem}
Let $N\geq 3$, $0<s_2<s_1<2$, $\lambda\in\Bbb R$ and $\Omega=\Bbb R^N_+$. Then equation\,\eqref{1equation} has a least-energy solution $u\in H^1_0(\Bbb R^N_+)$.
\end{thm}

\medskip

To complement Theorem\,\ref{2theorem}, we prove the following non-existence of entire solutions of \eqref{1equation} when $s_2=0$ and $\lambda\leq 0$.
\begin{thm}\label{3theorem}
Let $\Omega=\Bbb R^N_+$, $0<s_1\leq 2$ and $\lambda\leq 0$. 
Suppose $u(x) \in H^1_{loc}(\overline{\rn})$ and $u(x) \ge 0$ is a solution
of \eqref{3equation}. Then $u(x) \equiv 0$. 
\end{thm}

We note that if solutions are assumed to be in $H^1_0(\rn)$, 
Theorem~\ref{3theorem} with $s_1=2$ has been proved in \cite{HLW}. 
The authors of \cite{HLW} employed the method of moving planes to prove
Theorem~\ref{3theorem}, where the behavior of $u$ at $\infty$ is needed. One 
way to find asymptotic behavior is  to apply 
the Kelvein transform to $u$:
$$
\hat{u}(y) = \Big( \frac{1}{|y|} \Big)^{N-2} u \Big( \frac{y}{|y|^2} \Big)   \qquad {for } \qquad |y|<1.
$$
 It is a straightforward computation to show that $\hat{u}(y)$ satisfies
 $$
 \Delta \hat{u} + \lambda \frac{\hat{u}^{2^*(s_1) -1} (y) }{|y|^{s_1}} + \hat{u}^{\frac{N+2}{N-2}}(y) =0  \qquad \text{in }  \qquad  B_1 \cap \rn.
 $$
 But $\hat{u}$ is no longer contained in $H^1_{loc}(\overline{\rn})$, i.e., the integration of  $\nabla\hat{u}$ might be $+\infty$ in any neighborhood of $0$. In this case, the origin $0$ is called a nonremovable singularity of $\hat{u}$. It is a really interesting question : What is the asymptotic behavior of $\hat{u}$ near the singularity? Previously, this kind of problems have been studied:
 $$
 \Delta u + g(y, u)+ u^{\frac{N+2}{N-2}}=0  \qquad \text{ in  }  0<|y| <1.
 $$
  Under the monotonicity assumption of $u$:
  $$
  g(y, t) t^{-\frac{N+2}{N-2}}\,\, \text{ is decreasing for large } t >0,
  $$
 it was proved that $u(y) = O (|y|^{-\frac{N-2}{2}})$ near $0$. See \cite{CL0, CL,CL2, CL3}. For our case, 
 $$g(y, u)= \lambda \frac{u^{2^*(s_1) - 1 }}{|y|^{s_1}}
\qquad  \text{and} \qquad \lambda < 0.$$
 Then  $g(y, t) t^{-\frac{N+2}{N-2}}$ is increasing in $t>0$. Hence,
  the methods in \cite{CL0, CL,CL2, CL3} can not work for our nonlinearity.
  We should address this asymptotic problem later.
  
  Our proof of Theorem~\ref{3theorem}  employs the idea of the method 
  of moving spheres, a variant of the method of moving planes. 
The method of moving planes  has been developed through the works by 
A.D. Alexandrov, Serrin \cite{S}
, and Gidas, Ni and Nirenberg \cite{GNN}. 
Here, we will not require any assumption on the behavior of solutions at $\infty$, by taking
  some advantage of the upper half space $\rn$, while compared to $\mathbb R^N$.  We think this proof might be useful in other problems also.
See \cite{LZhu, LZhang, JLX} for some related
results. 
 
This paper is organized as follows. In Section 2,
we will prove
Theorem 1.3 and a generalization of it.
In Section 3, we will employ a blowing-up argument to prove Theorem~\ref{2theorem}. This kind of arguments have been developed for studying the nonlinear equation involving the Sobolev critical exponent, see \cite{CL0,CL,CL2,CL3,HL}.   The existence of least-energy solutions of equation \eqref{1equation} with $0<s_2 <s_1<2$ are obtained in Section 4. In final section, we discuss a perturbed equation of equation \eqref{1equation} for the case $\lambda < 0$, $0=s_2 < s_1<2$.

\section{Nonexistence of Entire Solutions}
\hspace{.6cm}
In this section, we begin with a proof of Theorem~\ref{3theorem}. 
We first make a remark about regularity of $u(x)$. It is shown that $u \in C^{\alpha}(\overline{\rn})$, for any $\alpha \in (0,1)$. For a proof, see \cite{CLW} and \cite{HLW}.
 
If $u=0$ at some point of $\Bbb R^N_+$, then $u\equiv 0$ by the strong maximum principle.
Hence, we will always assume that
\begin{equation}
u(x)>0.
\end{equation}
We will prove a lemma below.

\begin{lem}\label{lem31}
Let $u(x)$ be a positive solution of equation\,\eqref{3equation}. Suppose $u\in H^1_{loc}(\overline{\Bbb R^N_+})$. Then $\frac{\partial u}{\partial x_N}>0$ in $\Bbb R^N_+$.
\end{lem}

Before giving a proof of Lemma\,\ref{lem31}, we apply Lemma\,\ref{lem31} to prove Theorem\,\ref{3theorem}.

\medskip

\noindent
{\bf Proof of Theorem\,\ref{3theorem}.}
Suppose $u(x)$ is a positive solution of equation\,\eqref{3equation}.

We claim $u$ is uniformly bounded in any compact set of $\mathbb R^N$. 
Suppose the contrary,  then there exist
 $ \bar x_i \in \mathbb R^N$, such that
$$
u(\bar x_i)\to \infty\quad \mbox{as}\ i\to\infty.
$$
By the monotonicity of $u$ in $x_N-$direction, we may assume that
$$
|\bar x_i|\to\infty.
$$
Consider
$$
v_i(x)=(1-|x-\bar x_i|)^{ \frac {N-2}2} u(x), \qquad |x-\bar
x_i|\le 1.
$$
For some $|x_i-\bar x_i|<1$,
$$
v_i(x_i)=\max_{ |x-\bar x_i|\le 1}v_i(x),
$$
here we have used the fact that $v_i(x)=0$ for $|x-\bar
x_i|=1$.

Let
$$
\sigma_i=\frac 12(1-|x_i-\bar x_i|)>0.
$$
Then
$$
(2\sigma_i)^{ \frac {N-2}2} u(x_i)=  v_i(x_i)\ge v_i(\bar
x_i)=u(\bar x_i)\to \infty.
$$
It follows that
$$
R_i:= \sigma_i u(x_i)^{ \frac 2{N-2} }\to \infty.
$$

Since
$$
v_i(x_i)\ge v_i(x)\ge \sigma_i^{ \frac {N-2}2}u(x), \qquad
\forall\ x\in B_{ \sigma_i}(x_i),
$$
we see that
$$
u(x)\le 2^{ \frac{N-2}2}u(x_i),\qquad \forall\ x\in B_{
\sigma_i}(x_i).
$$

Consider
$$
w_i(y):=\frac 1{  u(x_i)  } u\left(x_i+\frac y{  u(x_i)^{
\frac 2{N-2} }   }\right),\qquad |y|<R_i = \sigma_i u_i(x_i)^{ \frac
2{N-2} }\to\infty.
$$
Then
$$
w_i(0)=1,\quad \mbox{and} \quad  w_i(y)\le 2^{ \frac {N-2}2},\ \
\forall\ |y|<R_i.
$$
Using the equation satisfied by $u_i$, we have
$$
\Delta w_i(y)-  \frac 1{  u(x_i)^{  \frac{2s_1}{N-2} }
} \frac {    w_i(y)^{  2^*(s_1)-1}  } {  |x_i+  \frac {y} {
u(x_i) ^{  \frac 2{N-2}}   }|^{s_1} } +w_i(y)^{ \frac {N+2}{N-2}
}=0,\qquad |y|<R_i.
$$

Since $|\bar x_i|\to\infty$, it is clear that
$$
 |x_i+  \frac {y} {
u(x_i) ^{  \frac 2{N-2}}   }|\ge |\bar x_i|-|x_i-\bar x_i| -
\sigma_i\ge |\bar x_i|-2\to\infty,\quad \mbox{uniformly for}\
|y|<R_i.
$$

Given the bound of $w_i$, we know from standard elliptic estimates
that on every compact subset of $\mathbb R^N$, $\{w_i\}$ is bounded in $C^3$
norm.   After passing to a subsequence, we have,
$$
w_i\to w\qquad \mbox{in}\ C^2_{loc}( \overline{\rn}).
$$
Given the above estimates, and the equation of $w_i$, we have
$$
\Delta w+w^{ \frac {N+2}{N-2} }=0,\qquad \mbox{on}\ \mathbb R^N,
$$
and
$$
w(0)=1, \quad w\ge 0\qquad \mbox{on}\ \mathbb R^N.
$$
By the strong maximum principle, $w>0$ on $R^N$.

By the classification theorem of Caffarelli-Gidas-Spruck,
\begin{equation}
\label{33equation}
w(y)=C_N\left(  \frac {\mu}{  1+\mu^2|y-y_0|^2  }\right)^{ \frac
{N-2}2 },
\end{equation}
where $\mu>0$ and $y_0\in \mathbb R^N$.

But we know from the monotonicity of $w_i$, $w$ must be monotone in
$y_N$-direction.  This is a contradiction and the claim is proved.

Let $u_j(x',x_N)=u(x', x_N+r_j)$ where $r_j<r_{j+1}\to+\infty$ as $j\to+\infty$.
By Lemma\,\ref{lem31}, $u_j(x)<u_{j+1}(x)$.  
 Since $u(x)$ is uniformly bounded, $u_j(x)\to u_\infty(x)$ in $C^2_{loc}(\Bbb R^N)$,
where $u_\infty(x)$ is a positive solution to
\begin{equation*}
\Delta u_\infty+u_\infty^{\frac{N+2}{N-2}}=0\quad\text{in \,}\Bbb R^N.
\end{equation*}
Again,  Lemma\,\ref{lem31} yields a 
 contradiction to \eqref{33equation}. Hence, the proof of Theorem~\ref{3theorem} is complete.
\hfill{$\square$}

{\bf Proof of Lemma~\ref{lem31}.}
 The 
proof uses  the method of moving spheres, a variant of
the method of moving planes which are developed through the works of
Alexandrov, Serrin \cite{S}, and Gidas, Ni and Nirenberg
\cite{GNN}.  We also make use of the ``narrow domain idea'' from
Berestycki and Nirenberg \cite{BN1}.

Define $$
x_R:= (0, \cdots, 0, -R).
$$
Let
$$
u_{x_{R}, \lambda}(y):= \left( \frac {\lambda}{  |y-x_R|  }\right)^{ N-2}
u\left(x_R+\frac {\lambda^2(y-x_R) }{ |y-x_R|^2}\right)
$$
be the Kelvin transformation of $u$ with respect to the ball
$B_\lambda(x_R)$ with center $x_R$ and radius $\lambda>0$. By 
direct computations, we have for $y \in B_{\lambda}(x_R)\cap \rn$,
$$
\Delta u_{x_R, \lambda}(y) - \Big(\frac{\lambda}{|y-x_R|}\Big)^{2s}\, \frac{u^{2^*(s)-1}_{x_R, \lambda}(y)}{|x_R + \frac{\lambda^2(y - x_R)}{y- x_R}|^s} + u_{x_R, \lambda}^{\frac{N+2}{N-2}}=0.
$$
We want to show that 
\begin{equation}
\label{2.5}
u_{x_R,\lambda} (y) \ge u(y)  \qquad \forall y \in B_{\lambda}(x_R)\cap \rn, \,\, \forall \lambda > R.
\end{equation}
To prove \eqref{2.5}, we first claim 
\begin{equation}
\label{2.6}
\left(\frac{\lambda}{|y-x_R|} \right)^{2s} \frac{1}{|x_R + \frac{\lambda^2 (y- x_R)}{|y - x_R|^2}|^s} \le \frac{1}{|y|^s},
\end{equation}
for $y \in B_{\lambda}(x_R) \cap \rn$, $\forall \lambda > R$.

For $y\in B_\lambda(x_R)\cap \rn$, we write
$$
\frac{ y-x_R}{ |y-x_R|}=\theta=(\theta_1, \cdots, \theta_N),\qquad
|y-x_R|=r.
$$
Then
\begin{equation}
\mu_1(\theta)<r<\lambda, \label{9-5} \end{equation} where
$\mu_1(\theta)$ is determined by
$$
x_R+\mu_1(\theta)\theta\in \partial \partial \rn.
$$
Namely,
$$
\mu_1(\theta)= \frac R{ \theta_N}.
$$

\bigskip

\eqref{2.6} is equivalent to
$$
\left(\frac \lambda r\right)^{2s} \frac 1{  |x_R+\frac{\lambda^2}r
\theta|^s}\le \frac 1{ |x_R+r\theta|^s}.
$$
This is equivalent to
\begin{equation}
\left( \frac{\lambda^2}r\right)^2
 \frac 1{  |x_R+\frac{\lambda^2}r
\theta|^2}\le r^2\frac 1{ |x_R+r\theta|^2}. \label{9-4}
\end{equation}
For $r$ satisfying (\ref{9-5}), we have
$$
\frac {\lambda^2}R>\frac {\lambda^2 }r>r>\mu_1(\theta).
$$
Let
$$
\eta(\mu):= \mu^2 \frac 1{  |x_R+\mu \theta|^2},\qquad  \mu >\mu_1(\theta).
$$
In order to prove (\ref{9-4}), we only need to prove
\begin{equation}
\eta'(\mu)\le 0,\qquad \mu_1(\theta)< \mu <\frac {\lambda^2}R. \label{11-1}
\end{equation}
This follows from the following calculations, for $\mu > \mu_1(\theta)$,

\begin{eqnarray*}
|x_R+ \mu \theta|^{4} \eta'(\mu)&=& 2 \mu |x_R+ \mu \theta|^2 - \mu^2 \frac d{d \mu}
(|x_R+ \mu \theta|^2)\\
&=& 2 \mu R \theta_N ( \mu_1 - \mu)<0.
\end{eqnarray*}
We have proved (\ref{11-1}), and therefore proved \eqref{2.6}. 
It follows that
$$
-\Delta u_{x_R, \lambda}+\frac 1{ |y|^s }u^{2(s)-1}_{x_R, \lambda}\ge u_{x_R,
\lambda}(y)^{ \frac {N+2}{N-2} },\qquad\mbox{in} \
B_\lambda(x_R)\cap
R^n_+.
$$
Thus
\begin{equation}
-\Delta \left(u_{x_R, \lambda}-u\right) +\frac 1{ |y|^s }
\left(u^{2^*(s)-1}_{x_R, \lambda}-u^{2^*(s)-1}\right)\ge
 u_{x_R,
\lambda}^{ \frac {N+2}{N-2} }-u^{ \frac {N+2}{N-2} },
\qquad\mbox{in} \ B_\lambda(x_R)\cap \rn. \label{9} \end{equation}
Write
\begin{equation}
w_\lambda=u_{x_R, \lambda}-u, \quad w_\lambda^-=\max\{0,
-w_\lambda\}. \label{notation} \end{equation}
 We first require that
$R<\lambda_0(R)<2R$, then for $R<\lambda<\lambda_0(R)$, we have
$$
|x_R+\frac {\lambda^2(y-x_R) }{ |y-x_R|^2}|\le |x_R|+ \frac
{\lambda^2}R\le 5R,\qquad \forall\ y\in B_\lambda(x_R)\cap \rn.
$$
 Multiply $w_\lambda^-$ to
the  inequality (\ref{9})  and integrate by parts on $B_\lambda(x_R)\cap
\rn$, we have, using $w_\lambda\ge 0$ on
$\partial(B_\lambda(x_R)\cap \rn)$,
\begin{eqnarray*}
&&\int_{B_\lambda(x_R)\cap \rn} |\nabla w_\lambda^-|^2dy  \\
&\le&\int_{B_\lambda(x_R)\cap \rn}(|\nabla w_\lambda^-|^2 - \frac 1{
|y|^s } (u^{2^*(s)-1}_{x_R, \lambda}-u^{2^*(s)-1}) w_\lambda^- dy \\
&\le &\int_{B_\lambda(x_R)\cap \rn} \left( u_{x_R, \lambda}^{
\frac {N+2}{N-2} }-u^{ \frac {N+2}{N-2} }\right) w_\lambda^-  dy \\
&\le & \frac {N+2}{N-2} \int_{B_\lambda(x_R)\cap
\rn}\left(\max\{u_{x_R, \lambda}, u\}\right)^{ \frac 4{N-2} }
(w_\lambda^-)^2 dy \\
&\le &\frac {N+2}{N-2} \sup_{  B_{5R}(0)\cap \rn } u^{ \frac
4{N-2} }\int_{B_\lambda(x_R)\cap \rn}(w_\lambda^-)^2 dy \\
&\le& C(N) |B_\lambda(x_R)\cap \rn|^{\frac 2N} \|w_\lambda^-\|_{
L^{ \frac{2N}{N-2}  }(B_\lambda(x_R)\cap \rn) }^2 \\
&\le & C(N)|B_\lambda(x_R)\cap \rn|^{\frac 2N}
\int_{B_\lambda(x_R)\cap \rn}|\nabla w_\lambda^-|^2 dy.
\end{eqnarray*}
Now we can choose $\lambda_0(R)>R$ but very close to $R$, then
$|B_\lambda(x_R)\cap \rn|$ is small, and we have
$$\int_{B_\lambda(x_R)\cap \rn}|\nabla w_\lambda^-|^2  dy \le \frac 12 \int_{B_\lambda(x_R)\cap
\rn}|\nabla w_\lambda^-|^2 dy.
$$
This implies $\nabla w_\lambda^-=0$ in $B_\lambda(x_R)\cap \rn$
and therefore, since $w_\lambda^-=0$ on $\partial B_\lambda(x_R)\cap
\rn$, $w_\lambda^-=0$ in $B_\lambda(x_R)\cap \rn$. Step 1 is
established.

\bigskip

Define
$$
\bar \lambda(R):= \sup\{\mu\ |\ \mu>R, \ \mbox{and}\ u_{x_R,
\lambda}(y)\ge u(y),\ \forall\ y\in B_\lambda(x_R)\cap \rn,\
\forall\ R<\lambda<\mu\}.
$$
By Step 1, $\bar \lambda(R)$ is well defined and $R<\bar
\lambda(R)\le \infty$.

\bigskip

{\bf Step 2.}\ $\overline{\lambda}(R)=\infty$ for all $R>0$.

\bigskip

We establish Step 2 by contradiction.  Suppose that $\bar
\lambda\equiv \bar\lambda(R)<\infty$ for some $R>0$.  Then
$$
 u_{x_R,
\bar \lambda}(y)\ge u(y),\ \forall\ y\in B_{\bar \lambda}(x_R)\cap
\rn.
$$
Since $u_{x_R, \bar \lambda}>u$ on $B_{\bar \lambda}(x_R)\cap
\partial \rn,$ we have, by the strong maximum principle,
$$
 u_{x_R,
\bar \lambda}(y)> u(y),\ \forall\ y\in B_{\bar \lambda}(x_R)\cap
\rn.
$$
For $\delta>0$ small, and the value to be fixed below, let
$$
K:=\{y\in B_{\bar \lambda}(x_R)\cap \rn\ |\ dist(y, \partial
(B_{\bar \lambda}(x_R)\cap \rn))\ge \delta\}.
$$
Then
$$
b:=\min_{  K } w_{\bar \lambda}>0,
$$
where we have used the notation (\ref{notation}).

Consider $\bar \lambda<\lambda<\bar\lambda+\epsilon$, where the
value of $\epsilon=\epsilon(\delta)<\delta$ is chosen so that
\begin{equation}
w_\lambda>\frac b2,\qquad\mbox{on}\ K,\qquad\forall\ \bar
\lambda<\lambda<\bar\lambda+\epsilon. \label{99} \end{equation}

Multiplying (\ref{9}) by $w_\lambda^-$ and integrating by parts on
$(B_\lambda(x_R)\cap \rn)\setminus K$ leads to, as before,
\begin{eqnarray*}
\begin{aligned}
& \int_{(B_\lambda(x_R)\cap \rn)\setminus K} |\nabla w_\lambda^-|^2 dy \\
& \le \int_{(B_\lambda(x_R)\cap \rn)\setminus K} (|\nabla w_\lambda^-|^2 - \frac 1{|y|^s } (u^{2^*(s)-1}_{x_R, \lambda}-u^{2^*(s)-1}) w_\lambda^- dy  \\
& \le C |(B_\lambda(x_R)\cap
\rn)\setminus K|^{\frac 2N} \int_{(B_\lambda(x_R)\cap
\rn)\setminus K}|\nabla w_\lambda^-|^2 dy .
\end{aligned}
\end{eqnarray*}
Now we can fix the value of $\delta$ so that $C |(B_\lambda(x_R)\cap
\rn)\setminus K|^{\frac 2N}<\frac 12$, and we obtain as before
$w_\lambda^-=0$ on $(B_\lambda(x_R)\cap \rn)\setminus K$, i.e.
$$
u_{x_R,  \lambda}(y)\ge u(y),\ \forall\ y\in \ (B_\lambda(x_R)\cap
\rn)\setminus K,\quad \forall\ \bar
\lambda<\lambda<\bar\lambda+\epsilon.
$$
This and (\ref{99}) contradicts to the definition of $\bar
\lambda(R)$. Step 2 is established.

\bigskip

By Step 2, we have
\begin{equation}
u_{x_R, R+a}(y)\ge u(y),\quad \forall\ y\in B_{R+a}(x_R)\cap \rn,\
\forall\ R,a>0. \label{4-1}
\end{equation}
It follows, for every $y\in \rn$, and every $a>y_n$,
$$
u(y)\le \lim_{R\to\infty}u_{x_R, R+a}(y)= u(y_1, \cdots, y_{N-1},
2a-y_N).
$$
The above implies
$$
u(y_1, \cdots, y_{N-1}, s)\le u(y_1, \cdots, y_{N-1}, t),\qquad
\forall\ 0<s<t.
$$
We have proved
$$
\frac{\partial u}{ \partial x_N}\ge 0,\qquad\mbox{in}\ \rn.
$$
Applying $\frac {\partial }{\partial x_N}$ to the equation of $u$
leads to
$$
-\Delta (\frac{\partial u}{ \partial x_N}) +\left( \frac
{2^*(s)-1}{|x|^s} u^{2^*(s)-2} -\frac {N+2}{N-2} u^{ \frac 4{n-2}} \right)(\frac{\partial
u}{ \partial x_N})+ \frac {\partial}{ \partial x_N}(\frac 1{
|x|^s}) u=0,\quad\mbox{in}\ \rn.
$$
By the strong maximum principle, we have $\frac{\partial u}{ \partial
x_n}$ is always zero or strictly positive.  But $u=0$ on the boundary on
$\rn$ and positive in $\rn$, so we must have $\frac{\partial u}{
\partial x_n}>0$ in $\rn$.  Lemma~\ref{lem31} is established. \hfill{$\square$}

The main theorem in this section is the following generalization of Theorem 1.3.
\begin{thm}\label{sec3thm}
Let $s_i \in(0,2]$, $P_i \in \mathbb R^{N-1}$ and let $u(x)\geq 0$ be a solution of
\begin{equation}\label{31equation}
\left\{
\begin{aligned}
&\Delta u-\sum_{i =1}^l\frac{u^{2^*(s_i)-1}}{|x-P_i|^{s_i}}+u^{\frac{N+2}{N-2}}=0\quad\text{in \,}\Bbb R^N_+,\\
&u(x)=0\quad\text{on \,}\partial\Bbb R^N_+.
\end{aligned}
\right.
\end{equation}
Suppose $u\in L^\infty(\Bbb R^N_+)\cap H^1_{loc}(\overline{\Bbb R^N_+})$. Then $u(x)\equiv 0$.
\end{thm}

{\bf Proof.}  The main step is to show that $\frac{\partial u}{\partial x_N} \ge 0$
as did in Lemma~\ref{lem31}. 
This second proof could work for the general situation of \eqref{31equation}, but the boundedness
of $u$ is required! The proof is divided into several steps.

{\bf Step 1.} $u(x) \to 0$ as $|x| \to +\infty$.
Suppose not. We may assume there are $x_j \to + \infty$, $u(x_j) \ge C >0$ for some positive
constant $C$. Let $u_j(x)=u(x+x_j)$. By elliptic estimates, $u_j(x)$ is bounded in $C^2$ in
any compact set of $\overline{\rn}$. By passing to a subsequence, we may assume $u_j(x) \to u(x)$
in $C^2_{loc}(\overline{\rn})$ and $u(x)$ satisfies
\begin{equation}
\label{2.14}
\left\{
\begin{aligned}
& \Delta u(x) + u^{\frac{N+2}{N-2}} =0   \qquad \text{ in } \rn, \\
& u \equiv 0  \qquad \text{ on } \partial \rn.
\end{aligned}
\right.  
\end{equation}

But it is well-known that \eqref{2.14} has no positive solutions. Thus, $u \equiv 0$ in $\rn$ which
contradicts to $u(0) \ge C >0$. So, Step 1 is proved.

\medskip

{\bf Step 2.} We claim for any $\lambda > 0$,
$$
u(y^\lambda) > u(y) \qquad \text{ for } x \in \Sigma_{\lambda} = \Big \{(y_1,y_2, \cdots, y_N) \Big|\,\, 0 \le y_N <\lambda \Big\},
$$
where $y^\lambda = (y_1, \cdots, y_{N-1}, 2\lambda - y_N )$.
 This step is a standard application of the method of moving planes. We give a sketch of 
 proofs for the sake of completeness. Let
 $$
 w_\lambda(y) = u(y^\lambda) - u(y).
 $$
 Then we have 
 \begin{eqnarray*}
 \begin{aligned}
 \Delta w_{\lambda}(y)& - \sum_{j=1}^{l} \frac{1}{|y-P_j|^{s_j}} \Big(u^{2^*(s_j)-1}(y^\lambda) - u^{2^*(s_j)-1}(y) \Big) + u^{\frac{N+2}{N-2}}(y^\lambda) - u^{\frac{N+2}{N-2}}(y)\\
  &= \sum_{j=1}^{l} \Big(\frac{1}{|y^\lambda - P_j|^{s_j}}-\frac{1}{|y|^{s_j}}\Big)u^{2^*(s_j)-1}(y^\lambda ) \le 0  \qquad \text{ in } \Sigma_{\lambda}.
 \end{aligned}
 \end{eqnarray*}
 Thus $w_\lambda (y)$ satisfies 
 $$
 \Delta w_\lambda(y)+(C_1(y) + C_2(y) ) w_\lambda(y) \le 0 \qquad \text{ in } \Sigma_\lambda,
 $$
 where 
 \begin{eqnarray*}
 C_1(y) \le 0, \qquad \text{ and } \qquad C_2(y) = \frac{u^{\frac{N+2}{N-2}}(y^{\lambda})-u^{\frac{N+2}{N-2}}(y)}{u(y^\lambda) - u(y)}.
 \end{eqnarray*}
 By Step 1, $C_2(y) = o(1)$ as $|y| \to + \infty$ and $y \in \Sigma_\lambda$.
 
 To prove $w_{\lambda}(y)>0$ in  $\Sigma_{\lambda}$ for $\lambda$ small, we consider
 the comparison function,
 $$
 v(y) = 1- y_N^2, \qquad 0\le y_N \le \lambda,
 $$
 and let
 $$
 \overline{w}_{\lambda}(y) = \frac{w_\lambda(y)}{v(y)}, \qquad \text{i.e.},\,\, w_\lambda(y) = \overline{w}_{\lambda}(y) v(y).  
 $$
Thus, $\overline{w}_{\lambda}$ satisfies
\begin{equation}
\label{2.15}
\left\{
\begin{aligned}
& \Delta \overline{w}_{\lambda}(y) + 2 \frac{\nabla v(y)}{v(y)}\cdot \nabla \overline{w}_{\lambda}(y)
+ \Big(C_1(y) + C_2(y) - \frac{4}{v(y)}\Big) \overline{w}_{\lambda}(y) \le 0, \\
& \overline{w}_{\lambda}(y', 0)>0 \qquad \text{ and } \qquad \overline{w}_\lambda(y', \lambda)=0.
\end{aligned}
\right.
\end{equation} 
Choose $\lambda$ small such that 
$$
\frac{N+2}{N-2} u^{\frac{4}{N-2}}(y) \le 2, \qquad 0 \le y_N \le \lambda.
$$
Now suppose the set $ \Big\{ y \Big|\,\, w_{\lambda}(y) < 0   \Big\} \ne \emptyset$.

 Because $\overline{w}_{\lambda} \ge 0$ on $\partial \Sigma_{\lambda}$ and $\lim_{|y| \to +\infty} \overline{w}_{\lambda}(y)=0$,
it is easy to see the minimum of $\overline{w}_{\lambda}$ can be achieved. 
Let 
$\overline{y}\in \Sigma_\lambda$ such that 
$$
\overline{w}_\lambda(\overline{y}) = \inf_{ y \in \Sigma_{\lambda}} \overline{w}_{\lambda}(y) < 0. 
$$
 Since $\overline{w}_{\lambda}(\overline{y})<0$,
$$
C(\overline{y}) \le \frac{N+2}{N-2} u^{\frac{4}{N-2}}(\overline{y}) \le 2.
$$
By applying the maximum principle, \eqref{2.15} yields
$$
0< \Big(C_1(\overline{y}) + C_2(\overline{y}) - \frac{4}{v(\overline{y})} \Big) \overline{w}_{\lambda}(\overline{y}) \le 0,
$$
which is a contradiction. Hence, $w_{\lambda}(y)>0$ $\forall y \in \Sigma_{\lambda}$.

Let 
$$
\overline{\lambda} = \sup \Big\{ \lambda \Big| \,\, w_{\mu}(y) > 0 \,\,\, \forall y \in \Sigma_{\mu},
0 < \mu \le \lambda \Big\}.
$$
We claim $\overline{\lambda} = +\infty $. Otherwise, we have 
\begin{equation}
\label{2.16}
\left\{
\begin{aligned}
& w_{\overline{\lambda}}(y) > 0 \qquad \forall y \in \Sigma_{\overline{\lambda}},\\
& \frac{\partial w_{\overline{\lambda}}}{\partial y_N}(y', \lambda) <0. 
\end{aligned}
\right.
\end{equation}
by the strong maximum principle and Hopf boundary point lemma. 
By the definition of $\overline{\lambda}$, there are $\lambda_j \downarrow \overline{\lambda}$
such that $\big\{  y \big|\, w_{\lambda_j(y)}<0, \,\, y\in \Sigma_{\lambda_j}  \big\} \ne \emptyset$.
Set
$$
w_{\lambda_j}(y) = \overline{w}_{\lambda_j}(y)v(y),
$$
where
 $$
  v(y)=(\overline{\lambda}+1)^2 - y_N^2, \,\, y \in \Sigma_{\lambda_j}.
$$
Then $\overline{w}_{\lambda_j}$ satisfies
\begin{equation}
\label{2.17}
\begin{aligned}
\Delta \overline{w}_{\lambda_j}(y) & +  2 \nabla \log v(y) \cdot \nabla \overline{w}_{\lambda_j}(y)\\
& +\Big(C_1(y) +C_2(y) -\frac{4}{v(y)}\Big) \overline{w} _{\lambda_j}\le 0
\end{aligned}
\end{equation}
Suppose $\overline{w}_{\lambda_j}(\overline{y}_j)= \inf_{ y \in \Sigma_{\lambda_j}} \overline{w}_{\lambda_j}(y) < 0.$ By \eqref{2.16}, we have $|\overline{y}_j| \to +\infty$. 
Note that 
$$
C_2(\overline{y}_j) \le \frac{N+2}{N-2} u ^{\frac{4}{N-2}}(\overline{y}_j) \to 0 \qquad \text{ as }
j \to +\infty. 
$$
 Again, by the maximum principle, \eqref{2.17} yields a contradiction. Therefore, Step 2 is proved.
 Obviously, the conclusion of Lemma~\ref{lem31} follows immediately from Step 2.  
\hfill{$\square$} 
After the proof $\frac{\partial u}{\partial x_N} > 0$, it is clear from 
the last step in the proof of Theorem~\ref{3theorem}, that the conclusion
of Theorem~\ref{sec3thm}  holds.  \hfill{$\square$}
\begin{rem}
{\rm
If $P_i = P_j$ $\forall i \ne j$, then the proof of Lemma~\ref{lem31} still holds. Hence, in this case, the boundedness assumption is not necessary for the conclusion of Theorem~\ref{sec3thm}.
}
\end{rem}
\section{Existence of Entire Solutions}
\hspace{.6cm}
In this section, we will give a proof of Theorem\,\ref{2theorem}. To prove Theorem\,\ref{2theorem}, we choose a convex domain $\Omega$ with $0\in\partial\Omega$ and consider the following equation. For any small $\varepsilon>0$
\begin{equation}\label{21equation}
\begin{cases}
&\Delta u+\displaystyle\lambda\frac{u^{p_1(\varepsilon)}}{|x|^{s_1}}+\frac{u^{p_2(\varepsilon)}}{|x|^{s_2}}=0\quad\text{in \,}\Omega,\\
&u(x)>0\quad\text{in \,}\Omega\text{ \,and \,}u(x)=0\quad\text{on \,}\partial\Omega,
\end{cases}
\end{equation}
where $p_1(\varepsilon)=2^*(s_1)-1-\varepsilon$ and $p_2(\varepsilon)=2^*(s_2)-1-(\frac{2-s_2}{2-s_1})\varepsilon$.

\medskip

For $\varepsilon>0$, we let
\begin{equation*}
\varPhi_\varepsilon(u):=\frac{1}{2}\int_\Omega|\nabla u|^2dx-\frac{\lambda}{p_1(\varepsilon)+1}\int_\Omega \frac{ (u^+)^{p_1(\varepsilon)+1}}{|x|^{s_1}}dx-\frac{1}{p_2(\varepsilon)+1}\int_\Omega \frac{(u^+)^{p_2(\varepsilon)+1}}{|x|^{s_2}}dx
\end{equation*}
for $u\in H^1_0(\Omega)$ and
\begin{equation*}
c_\varepsilon^*=\displaystyle\inf_{P\in{\cal P}}\max_{w\in P}\varPhi_\varepsilon (w),
\end{equation*}
where ${\cal P}$ is the class of all continuous paths in $H^1_0(\Omega)$ connecting $0$ and some $u_0$ such that $\varPhi_\varepsilon(u_0)\leq 0$.
It is easy to see that $c^*_\varepsilon\leq C$ for some constant $C$ independent of $\varepsilon$.
Since for $\varepsilon>0$, $\varPhi_\varepsilon$ satisfies the P-S condition, it is known that $c^*_\varepsilon$ is a critical point of $\varPhi_\varepsilon$, i.e., there exists a solution $u_\varepsilon\in H^1_0(\Omega)$ with $\varPhi_\varepsilon(u_\varepsilon)=c^*_\varepsilon$. Thus,
\begin{equation}\label{22equation}
\begin{cases}
&\frac{1}{2}\int_\Omega|\nabla u_\varepsilon|^2dx-\frac{\lambda}{p_1(\varepsilon)+1}\int_\Omega \frac{u_\varepsilon^{p_1(\varepsilon)+1}}{|x|^{s_1}}dx-\frac{1}{p_2(\varepsilon)+1}\int_\Omega \frac{u_\varepsilon^{p_2(\varepsilon)+1}}{|x|^{s_2}}dx=c_\varepsilon^*,\\
&\int_\Omega|\nabla u_\varepsilon|^2dx-\lambda\int_\Omega\frac{u_\varepsilon^{p_1(\varepsilon)+1}}{|x|^{s_1}}dx-\int_\Omega\frac{u_\varepsilon^{p_2(\varepsilon)+1}}{|x|^{s_2}}dx=0.
\end{cases}
\end{equation}
From \eqref{22equation}, we have
\begin{equation*}
\left(\frac{1}{2}-\frac{1}{p_1(\varepsilon)+1}\right)\int_\Omega|\nabla u_\varepsilon|^2dx+\left(\frac{1}{p_1(\varepsilon)+1}-\frac{1}{p_2(\varepsilon)+1}\right)\int_\Omega\frac{u_\varepsilon^{p_2(\varepsilon)+1}}{|x|^{s_2}}dx=c_\varepsilon^*.
\end{equation*}
By noting both of $\frac{1}{2}-\frac{1}{p_1(\varepsilon)+1}$ and $\frac{1}{p_1(\varepsilon)+1}-\frac{1}{p_2(\varepsilon)+1}$
are positive, we have
\begin{equation}
\int_\Omega|\nabla u_\varepsilon|^2dx+\int_\Omega \frac{u_\varepsilon^{p_1(\varepsilon)+1}}{|x|^{s_1}}dx+\int_\Omega\frac{u_\varepsilon^{p_2(\varepsilon)+1}}{|x|^{s_2}}dx\leq C<+\infty.
\end{equation}
Therefore, by passing to a subsequence if necessary, we might assume $u_\varepsilon\rightharpoonup u$ in $H^1_0(\Omega)$ as $\varepsilon\to 0$.
If $u\not\equiv 0$, then $u$ is a solution of
\begin{equation}\label{equationcritical}
\begin{cases}
&\Delta u+\lambda\frac{u^{p_1}}{|x|^{s_1}}+\frac{u^{p_2}}{|x|^{s_2}}=0\quad\text{in \,}\Omega,\\
&u=0\quad\text{on \,}\partial\Omega.
\end{cases}
\end{equation}
However, the standard Pohozaev indentity yields that equation\,\eqref{equationcritical} has no positive solutions because both of $p_1$ and $p_2$ are critical exponents.
Thus $u\equiv 0$ and $u_\varepsilon(x)$ must blow up as $\varepsilon\to 0$.

\medskip

Before poceeding further, we will briefly discuss
the regularity of $u$ at $0$. Because $s_1<2$, we can prove that
\begin{equation}
\label{decayestimate}
\begin{cases}
&u\in C^2(\overline\Omega)\quad\text{if } s_1< 1+\frac{2}{n},\\
&u\in C^{1,\beta}(\overline\Omega)\quad\text{for all }0<\beta<1\quad\text{if }s_1=1+\frac{2}{n},\\
&u\in C^{1,\beta}(\overline\Omega)\quad\text{for all
}0<\beta<\frac{n(2-s)}{n-2}\quad\text{if } s_1> 1+\frac{2}{n},
\end{cases}
\end{equation}
see \cite{HLW}.

\medskip

Let
\begin{equation*}
u_\varepsilon(x_\varepsilon)=\max_{\overline\Omega}u_\varepsilon(x)=m_\varepsilon\text{ \,and \,}k_\varepsilon=m_\varepsilon^{-\frac{p_2(\varepsilon)-1}{2-s_2}}.
\end{equation*}
By direct computations, we have
\begin{equation*}
k_\varepsilon=m_\varepsilon^{-\frac{2}{N-2}+\frac{\varepsilon}{2-s_1}}.
\end{equation*}

First, we claim
\begin{equation}\label{firstclaim}
|x_\varepsilon|=O(k_\varepsilon).
\end{equation}
Suppose not. By passing to a subsequence if necessary, we may assume
\begin{equation*}
\lim_{\varepsilon\to 0}\frac{|x_\varepsilon|}{k_\varepsilon}=+\infty.
\end{equation*}
By scaling, we set
\begin{equation*}
\tilde v_\varepsilon(y)=\frac{u_\varepsilon(x_\varepsilon+r_\varepsilon y)}{m_\varepsilon}\quad\text{in \,}\Omega_\varepsilon,
\end{equation*}
where
\begin{equation*}
\Omega_\varepsilon=\{y\in\Bbb R^N\,|\, x_\varepsilon+ r_\varepsilon y\in\Omega\},
\quad\text{and}\quad r_\varepsilon=|x_\varepsilon|^{\frac{s_2}{2}}k_\varepsilon^{\frac{2-s_2}{2}}=(\frac{|x_\varepsilon|}{k_\varepsilon})^{\frac{s_2}{2}}k_\varepsilon.
\end{equation*}
By equation\,\eqref{21equation}, $\tilde v_\varepsilon(y)$ satisfies
\begin{equation*}
\begin{cases}
&\Delta \tilde v_\varepsilon+\lambda(\frac{k_\varepsilon}{|x_\varepsilon|})^{s_1-s_2}\frac{\tilde v_\varepsilon^{p_1(\varepsilon)}}{\left|y_\varepsilon+\frac{r_\varepsilon}{|x_\varepsilon|}y\right|^{s_1}}
+\frac{\tilde v_\varepsilon^{p_2(\varepsilon)}}{\left|y_\varepsilon+\frac{r_\varepsilon}{|x_\varepsilon|}y\right|^{s_2}}=0\quad\text{in \,}\Omega_\varepsilon,\\
&y_\varepsilon=\frac{x_\varepsilon}{|x_\varepsilon|}\quad\text{and}\quad \tilde v_\varepsilon(y)\leq \tilde v_\varepsilon(0)=1.
\end{cases}
\end{equation*}
Let $\Omega_\varepsilon\to H$ as $\varepsilon\to 0$, where either $H=\Bbb R^N$ or $H$ is a closed half space of $\Bbb R^N$.
Note $(\frac{k_\varepsilon}{|x_\varepsilon|})^{s_1-s_2}$ and $\frac{r_\varepsilon}{|x_\varepsilon|}=(\frac{k_\varepsilon}{|x_\varepsilon|})^{\frac{2-s_2}{2}}$ tend to $0$ as $\varepsilon\to 0$.
Then by applying elliptic estimates, $\tilde v_\varepsilon$ converges to $\tilde v$ in $C^2_{loc}(H)$, where
\begin{equation*}
\Delta \tilde v+\tilde v^{p_2}=0\quad\text{in \,}H.
\end{equation*}
If $H$ is a half space of $\Bbb R^N$, then $v$ also satisfies $v=0$ on $\partial H$. Since $p_2=\frac{2(N-s_2)}{N-2}-1<\frac{N+2}{N-2}$,
$v(y)\equiv 0$ in $H$ no matter $H$ is $\Bbb R^N$ or a half space. But it yields a contradiction to $v(0)=1$. Thus, the claim is proved.

\medskip

After \eqref{firstclaim} is established, we set
\begin{equation*}
v_\varepsilon(y)=m_\varepsilon^{-1}u_\varepsilon(x_\varepsilon+k_\varepsilon y).
\end{equation*}
Then $v_\varepsilon(y)$ satisfies
\begin{equation*}
\Delta v_\varepsilon +\lambda\frac{v_\varepsilon^{p_1(\varepsilon)}}{\left|\frac{x_\varepsilon}{k_\varepsilon}+y\right|^{s_1}}+\frac{v_\varepsilon^{p_2(\varepsilon)}}{\left|\frac{x_\varepsilon}{k_\varepsilon}+y\right|^{s_2}}
=0\quad\text{in \,}\Omega_\varepsilon,
\end{equation*}
where $\Omega_\varepsilon=\{y\in\Bbb R^N\,|\, x_\varepsilon+k_\varepsilon y\in\Omega\}$.
Since $\frac{x_\varepsilon}{k_\varepsilon}$ is bounded, without loss of generality, we may assume $\frac{x_\varepsilon}{k_\varepsilon}\to y_0$.
Therefore, $\Omega_\varepsilon\to H$ as $\varepsilon\to 0$, where $H$ is a half spae of $\Bbb R^N$ with $-y_0\in\partial H$
and by the elliptic estimates, $v_\varepsilon(y)\to v(y)$ in $C^2_{loc}(H)$.
Clearly, $v$ satisfies
\begin{equation}
\begin{cases}
&\Delta v+\lambda\frac{v^{p_1}}{|y_0+y|^{s_1}}+\frac{v^{p_2}}{|y_0+y|^{s_2}}=0\quad\text{in \,}H,\\
&v=0\quad\text{on \,}\partial H.
\end{cases}
\end{equation}
Since $v(0)=1$, we have $y_0\ne 0$. By a linear transformation of $y$, $H$ can be map onto $\Bbb R^N_+$ and $v$ is an entire solution of equation\,\eqref{1equation} with $\Omega=\Bbb R^N_+$.
This completes the proof of the existence part of Theorem\,\ref{2theorem}. \hfill{$\square$}

\medskip

\noindent
{\bf Remark.}
Suppose $v$ is a positive entire soltion of \eqref{1equation}. Then the Kelvin transformation $\hat v(y)=|y|^{2-n}v(\frac{y}{|y|^2})$
is also a positive entire solution.
By the regularity \eqref{decayestimate}, $|\hat{v}(y)| \le C|y|$ for $|y| < 1$. Thus,
\begin{equation}
\label{eqq3.8}
|v(y)| \le C |y|^{1-n}  \qquad \text{ for }\,\, |y| \ge 1.
\end{equation}
By the standard gradient estimate, we have 
\begin{equation}
\label{eqq3.9}
|\nabla v(y)| \le |y|^{-n}  \qquad \text{for }\,\,  |y| \ge 1. 
\end{equation}
 By using the well-known method of moving sphere, it can be proved that after a suitable scaling, $v(y)=\hat v(y)$.
Since the argument is standard now, the proof is omitted here.

\begin{cor}\label{21corollary}
There exists an entire solution $v$ of equation\,\eqref{1equation} with $\Omega=\Bbb R^N_+$ such that the critical value
$\varPhi(v)=\inf\{\varPhi(u)\,|\, u\text{ is an entire solution of }\eqref{1equation}\}$.
\end{cor}
\begin{proof}
We first note that 
$$
\Big\{ \Phi(u)\, \Big|\,  \text{u is a positive entire solution of \eqref{1equation}}\Big\}  \ne \emptyset, 
$$
because
\begin{eqnarray*}
\begin{aligned}
\int_{\rn} |\nabla u|^2 dy & = \lambda \int_{\rn} \frac{u^{2^*(s_1)}}{|x|^{s_1}} dy + \int_{\rn} \frac{u^{2^*(s_2)}}{|x|^{s_2}}dy \\
& \le C \Big[   \Big( \int_{\rn} |\nabla u|^2 dy \Big)^{\frac{2^*(s_1)}{2}} + 
\Big( \int_{\rn} |\nabla u|^2 dy  \Big)^{\frac{2^*(s_2)}{2}}   \Big]
\end{aligned}
\end{eqnarray*}
implies $ \| \nabla u \| \ge c_0 $ for some constant $c_0>0$.

Suppose $v_j$ is a sequence of positive entire solutions of
\begin{equation}
\begin{cases}
&\Delta v_j+\lambda\frac{v_j^{p_1}}{|y|^{s_1}}+\frac{v_j^{p_2}}{|y|^{s_2}}=0\quad\text{in \,}\Bbb R^N_+,\\
&v_j=0\quad\text{on \,}\partial\Bbb R^N_+
\end{cases}
\end{equation}
such that $\varPhi(v_j)\downarrow\inf\{\varPhi(u)\,|\, u\text{ is an entire solution of }\eqref{1equation}\}$.
By the remark above, we can assume $\hat v_j(y)=v_j(y)$. By \eqref{22equation} again, $\|\nabla v_j\|_{L^2(\Bbb R^N_+)}\leq C$ for some
constant $C>0$. Let $v_j\rightharpoonup v$ in $H^1_0(\Bbb R^n_+)$. If $v\ne 0$, then
\begin{align*}
\lim_{j\to +\infty}\varPhi(v_j)&=(\frac{1}{2}-\frac{1}{p_1+1})\lim_{j\to\infty}\int_{\Bbb R^N_+}|\nabla v_j|^2dy+(\frac{1}{p_1+1}-\frac{1}{p_2+1})\lim_{j\to+\infty}\int_{\Bbb R^N_+}\frac{v_j^{p_2+1}}{|y|^{s_2}}dy\\
&\geq(\frac{1}{2}-\frac{1}{p_1+1})\int_{\Bbb R^N_+}|\nabla v|^2dy+(\frac{1}{p_1+1}-\frac{1}{p_2+1})\int_{\Bbb R^N_+}\frac{v^{p_2+1}}{|y|^{s_2}}dy=\varPhi(v).
\end{align*}
Then it yields the conclusion of Corollary\,\ref{21corollary}.

\medskip

If $v_j\rightharpoonup 0$, then $\displaystyle\max_{|y|\leq 2}v_j(y)\to+\infty$, because $\hat v_j(y)=v_j(y)$ implies
\begin{align*}
\frac{1}{2}\varPhi(v_j)&=\frac{1}{2}\int_{B_1}|\nabla v_j|^2dy-\frac{\lambda}{p_1+1}\int_{B_1}\frac{v_j^{p_1+1}}{|y|^{s_1}}dy-\frac{1}{p_2+1}\int_{B_2}\frac{v_j^{p_2+1}}{|y|^{s_2}}dy\\
&=(\frac{1}{2}-\frac{1}{p_1+1})\int_{B_1}|\nabla v_j|^2dy+(\frac{1}{p_1+1}-\frac{1}{p_2+1})\int_{B_2}\frac{v_j^{p_2+1}}{|y|^{s_2}}dy.
\end{align*}

By the proof of Theorem\,\ref{2theorem}, we see that $v_j$ blows up at $y=0$ and the scaling $w_j(y)$:
\begin{equation*}
w_j(y)=\frac{v_j(x_j+k_j y)}{v_j(x_j)}\to w,
\end{equation*}
where $v_j(x_j)=\displaystyle\max_{|y|\leq 2}v_j(y)\to \infty$, $k_j=m_j^{-\frac{2}{N-2}}$, and $w$ is also a positive entire solution of equation\,\eqref{1equation}.
Thus,
\begin{equation*}
\lim_{j\to +\infty}\frac{1}{2}\varPhi(v_j)\geq (\frac{1}{2}-\frac{1}{p_1+1})\int_{\Bbb R^N_+}|\nabla w|^2dy+(\frac{1}{p_1+1}-\frac{1}{p_2+1})\int_{\Bbb R^N_+}\frac{w^{p_2+1}}{|y|^{s_2}}dy=\varPhi(w),
\end{equation*}
which yields
\begin{equation*}
\inf\{
\varPhi(u)\,|\, u\text{ is an entire solution of equation\,\eqref{1equation}}=0
\},
\end{equation*}
a contradiction. Hence, $v_j\not\rightharpoonup 0$, and Corollary\,\ref{21corollary} is proved.
\end{proof}

\section{Proof of Theorem\,\ref{1theorem}}
\hspace{.6cm}
Let $v(y)$ be a least-energy solution of
\begin{equation}\label{eqentvy}
\begin{cases}
&\Delta v+\lambda\frac{v^{p_1}}{|y|^{s_1}}+\frac{v^{p_2}}{|y|^{s_2}}=0\quad\text{in\ \,}\Bbb R^N_+,\\
&v(y)>0\quad\text{in \,}\Bbb R^N_+\quad\text{and}\quad v(y)=0\quad\text{on \,}\partial\Bbb R^N_+,
\end{cases}
\end{equation}
and
\begin{equation}
\label{eqqq4.2}
c_1=\varPhi(v)=\frac{1}{2}\int_{\Bbb R^N_+}|\nabla v|^2dy-\frac{\lambda}{2^*(s_1)}\int_{\Bbb R^N_+}\frac{v^{2^*(s_1)}}{|y|^{s_1}}dy
-\frac{1}{2^*(s_2)}\int_{\Bbb R^N_+}\frac{v^{2^*(s_2)}}{|y|^{s_2}}dy.
\end{equation}

\begin{lem}\label{lem41critical}
Suppose that $\Omega$ is a bounded smooth domain in $\Bbb R^N$ with $0\in\partial\Omega$. If $H(0)<0$, then there exists a nonnegative function $v_0\in H^1_0(\Omega)\setminus\{0\}$
such that
\begin{equation*}
\max_{t\geq 0}\varPhi(t v_0)<c_1.
\end{equation*}
\end{lem}
\begin{proof}
Without loss of generality, we may assume that in a neighborhood of $0$, $\partial\Omega$ can be represented by
$x_n=\varphi(x')$, where $x'=(x_1,\cdots, x_{N-1})$, $\varphi(0)=0$, $\nabla'\varphi(0)=0$, $\nabla'=(\partial_1,\cdots,\partial_{N-1})$,
and the outer normal of $\partial\Omega$ at $0$ is $-e_N=(0,\cdots,0,-1)$. Define
\begin{equation*}
\phi(x):=(x', x_N-\varphi(x')).
\end{equation*}
We choose a small positive number $r_0$ so that there exist neighborhoods of $0$, $U$ and $\tilde U$, such that
$\phi(U)=B_{r_0}(0)$, $\phi(U\cap\Omega)=B_{r_0}^+(0)$, $\phi(\tilde U)=B_{\frac{r_0}{2}}(0)$ and $\phi(\tilde U\cap\Omega)=B_{\frac{r_0}{2}}^+(0)$.
Here, we adopt the notation:
\begin{equation*}
B_{r_0}^+(0)=B_{r_0}\cap\Bbb R^N_+\quad\text{for \,}r_0>0.
\end{equation*}
Let $\eta\in C_0^\infty(U)$ be a positive cut-off function with $\eta\equiv 1$ in $\tilde U$. Set
\begin{equation*}
u_\e(x):=\eta(x)v_\e(x):=\eta(x)\e^{-\frac{N-2}{2}}v\left(\frac{\phi(x)}{\e}\right)\quad\text{for \,}x\in \Omega.
\end{equation*}
For $t\geq 0$, we have
\begin{equation}\label{eachintest}
\varPhi(t u_\e)=\frac{t^2}{2}\int_\Omega|\nabla u_\e|^2dx-\frac{\lambda t^{2^*(s_1)}}{2^*(s_1)}\int_\Omega\frac{u_\e^{2^*(s_1)}}{|x|^{s_1}}dx
-\frac{t^{2^*(s_2)}}{2^*(s_2)}\int_\Omega\frac{u_\e^{2^*(s_2)}}{|x|^{s_2}}dx\quad\text{for \,}u\in H^1_0(\Omega).
\end{equation}

In what follows, we estimate each integral on the right-hand side of \eqref{eachintest}.
Basically, the computation will be similar to Lemma\,2.2 in \cite{CLW}.
For the sake of completeness, we will sketch the proof here. We refer the readers to \cite{CLW} for details of computation.

\medskip

By the change of the variable $\frac{\phi(x)}{\e}=y$, we get
\begin{align*}
&\int_\Omega |\nabla u_\e|^2dx=\int_{U\cap\Omega}\eta^2|\nabla v_\e|^2dx-\int_{U\cap\Omega} \eta(\Delta \eta)v_\e^2dx\\
&=\int_{\Bbb R^N_+}|\nabla v(y)|^2dy-2\int_{B_{\frac{r_0}{\e}}^+}\eta\left(\phi^{-1}(\e y)\right)^2\partial_N v(y)\nabla'v(y)\cdot(\nabla'\varphi)(\e y')dy+O(\e^2).
\end{align*}
By using integration by parts and equation\,\eqref{eqentvy}, the second term can be estimated as the following.
\begin{align*}
&-2\int_{B_{\frac{r_0}{\e}}^+}\eta\left(\phi^{-1}(\e y)\right)^2\partial_N v(y)\nabla' v(y)\cdot(\nabla' \varphi)(\e y')dy\\
&=\frac{2}{\e}\int_{B_{\frac{r_0}{\e}}^+}\eta\left(\phi^{-1}(\e y)\right)^2\partial_N v(y)\sum_{i=1}^{N-1}\partial_{ii}v(y)\varphi(\e y')+O(\e^2)\\
&=-\frac{1}{\e}\int_{B_{\frac{r_0}{\e}}^+}\eta\left(\phi^{-1}(\e y)\right)^2\partial_N\left[(\partial_N v(y))^2\right]\varphi(\e y')dy\\
&-\frac{2\lambda}{2^*(s_1)\e}\int_{B_{\frac{r_0}{\e}}^+}\eta\left(\phi^{-1}(\e y)\right)^2\frac{\partial_N\left[v(y)^{2^*(s_1)}\right]}{|y|^{s_1}}\varphi(\e y')dy\\
&-\frac{2}{2^*(s_2)\e}\int_{B_{\frac{r_0}{\e}}^+}\eta\left(\phi^{-1}(\e y)\right)^2\frac{\partial_N\left[v(y)^{2^*(s_2)}\right]}{|y|^{s_2}}\varphi(\e y')dy+O(\e^2)\\
&=:I_1+I_2+I_3+O(\e^2).
\end{align*}
Since $\partial\Omega$ is $C^2$ at $0$, $\varphi$ can be expanded as
\begin{equation*}
\varphi(y')=\sum_{i=1}^{N-1}\alpha_i y_i^2+o(1)|y'|^2.
\end{equation*}
Hence,
\begin{align*}
I_1&=\frac{1}{\e}\int_{B_{\frac{r_0}{\e}}^+\cap\partial\Bbb R^N_+}\eta\left(\phi^{-1}(\e y')\right)^2\partial_Nv(y',0)\varphi(\e y')dy'\\
&=\e\sum_{i=1}^{N-1}\alpha_i\int_{\Bbb R^{N-1}}\left(\partial_N v(y',0)\right)^2y_i^2dy'(1+o(1))+O(\e^2)
=K_1 H(0)\e(1+o(1))+O(\e^2),
\end{align*}
where
\begin{equation*}
K_1=\int_{\Bbb R^{N-1}}|\partial_N v(y',0)|^2|y'|^2dy',
\end{equation*}
\begin{align*}
I_2&=-\frac{2\lambda s_1}{2^*(s_1)\e}\int_{B_{\frac{r_0}{\e}}^+}\eta\left(\phi^{-1}(\e y)\right)^2\frac{v(y)^{2^*(s_1)}y_N}{|y|^{2+s_1}}\varphi(\e y')dy\\
&=-K_2 H(0)(1+o(1))\e+O(\e^2),
\end{align*}
where
\begin{equation*}
K_2=\frac{2\lambda s_1}{2^*(s_1)}\int_{\Bbb R^N_+}\frac{v(y)^{2^*(s_1)}|y'|^2y_N}{|y|^{2+s_1}}dy,
\end{equation*}
and
\begin{align*}
I_3&=-\frac{2s_2}{2^*(s_2)\e}\int_{B_{\frac{r_0}{\e}}^+}\eta\left(\phi^{-1}(\e y)\right)^2\frac{v(y)^{2^*(s_2)}|y'|^2y_N}{|y|^{2+s_2}}\varphi(\e y')dy\\
&=-K_3H(0)\e (1+o(1))+O(\e^2),
\end{align*}
where
\begin{equation*}
K_3=\frac{2s_2}{2^*(s_2)}\int_{\Bbb R^N_+}\frac{v(y)^{2^*(s_2)}|y'|^2y_N}{|y|^{2+s_2}}dy.
\end{equation*}
By \eqref{eqq3.8} and \eqref{eqq3.9}, $K_i$, $i=1,2,3,$ are finite. Therefore,
 we have
\begin{equation*}
\int_\Omega|\nabla u_\e|^2dx=\int_{\Bbb R^N_+}|\nabla v|^2dy+\e H(0)(K_1-K_2-K_3)(1+o(1))+O(\e^2),
\end{equation*}
and similarly, we have
\begin{align*}
\lambda\int_\Omega\frac{u_\e^{2^*(s_1)}}{|x|^{s_1}}dx
&=\lambda\int_{\Bbb R^N_+}\frac{v^{2^*(s_1)}}{|y|^{s_1}}dy-\frac{s\lambda}{\e}\int_{B_{\frac{r_0/2}{\e}}^+}\frac{v(y)^{2^*(s_1)} y_N\varphi(\e y')}{|y|^{2+s_1}}dy+O(\e^2)\\
&=\lambda\int_{\Bbb R^N_+}\frac{v^{2^*(s_1)}}{|y|^{s_1}}dy-2^*(s_1)K_2 H(0)\e(1+o(1))+O(\e^2),
\end{align*}
and
\begin{equation*}
\int_\Omega\frac{u_\e^{2^*(s_2)}}{|x|^{s_2}}dx=\int_{\Bbb R^N_+}\frac{v^{2^*(s_2)}}{|y|^{s_2}}dy
-2^*(s_2)K_3 H(0)\e(1+o(1))+O(\e^2).
\end{equation*}
Therfore,
\begin{align*}
\varPhi(t u_\e)&=\frac{t^2}{2}\int_{\Bbb R^N_+}|\nabla v|^2dy-\frac{\lambda t^{2^*(s_1)}}{2^*(s_1)}\int_{\Bbb R^N_+}\frac{v^{2^*(s_1)}}{|y|^{s_1}}dy
-\frac{t^{2^*(s_2)}}{2^*(s_2)}\int_{\Bbb R^N_+}\frac{v^{2^*(s_2)}}{|y|^{s_2}}dy\\
&+\frac{t^2}{2}\left(\e H(0)(K_1-K_2-K_3)+o(1)\right)+\frac{t^{2^*(s_1)}}{2}K_2H(0)(1+o(1))\\
&+\frac{t^{2^*(s_2)}}{2}K_3H(0)(1+o(1))\e+O(\e^2)=:f_1(t)+f_2(t)+O(\e^2),
\end{align*}
where
\begin{equation*}
f_1(t)=\frac{t^2}{2}\int_{\Bbb R^N_+}|\nabla v|^2dy-\frac{\lambda t^{2^*(s_1)}}{2^*(s_1)}\int_{\Bbb R^N_+}\frac{v^{2^*(s_1)}}{|y|^{s_1}}dy
-\frac{t^{2^*(s_2)}}{2^*(s_2)}\int_{\Bbb R^N_+}\frac{v^{2^*(s_1)}}{|y|^{s_1}}dy.
\end{equation*}
Since $2^*(s_2)>2^*(s_1)$, $\varPhi(t u_\e)$ has the unique maximum. Note that
\begin{equation*}
\max_{t\geq 0}f_1(t)=f_1(1)=c_1.
\end{equation*}
Hence, the maximum of $\varPhi(t u_\e)$ occurs at $t_\e=1+o(1)$.
By noting that
\begin{equation*}
f_2(t)=\e H(0)\left[
\frac{t^2}{2}(K_1-K_2-K_3+o(1))+\frac{t^{2^*(s_1)}}{2}K_2(1+o(1))+\frac{t^{2^*(s_2)}}{2}K_3(1+o(1))
\right],
\end{equation*}
and $f_2(1)=\e H(0)K_1(1+o(1))<0$. Hence, we have
\begin{equation*}
\max_{t\geq 0}\varPhi(t u_\e)=\varPhi(t_\e u_\e)\leq f_1(t_\e)+f_2(t_\e)<f_1(t_\e)\leq f_1(1)=c_1.
\end{equation*}
Thus Lemma\,\ref{lem41critical} is proved.
\end{proof}

We are now in a position to prove Theorem~\ref{1theorem}

{\bf Proof of Theorem~\ref{1theorem}.}
As before, we let for small positive $\e$, 
$$
\Phi_{\e}(u) = \frac{1}{2} \int_{\Omega} |\nabla u|^2 dx - \frac{\lambda}{ p_1(\e)+1} \int_{\Omega} \frac{(u^+)^{p_1(\e) +1 }}{|x|^{s_1}} dx -\frac{1}{ p_2(\e) +1}\int_{\Omega} \frac{(u^+)^{p_2(\e) +1} }{|x|^{s_2}} dx,
$$
where $p_1(\e) = 2^*(s_1)-1-\e$ and $p_2(\e)=2^*(s_2)-1-\Big( \frac{2-s_2}{2-s_1}\Big) \e$. By Lemma~\ref{lem41critical}, there exists $u_0 \in H^1_0(\Omega)$ such that $\Phi_{\e}(u_0)\le 0$ and 
$$
c^*_{\e} = \inf_{P\in \cal P } \max_{w \in P} \Phi_{\e}(w),
$$
where $\cal P$ is the class of all continuous paths in $H^1_0(\Omega)$ connecting $0$ with $u_0$. Let $c_1$ be the positive constant defined by \eqref{eqqq4.2}. Then Lemma~\ref{lem41critical} yields
\begin{equation}
\label{eqq4.4}
c^*_{\e} < c_1,
\end{equation}
provided that $\e \in [0, \e_0]$ for some small $\e_0>0$.  For $\e >0$, there
exists a solution $u_{\e} \in H^1_0(\Omega)$ of 
\begin{eqnarray*}
\begin{aligned}
& \Delta u_{\e} + \lambda \frac{u^{p_1(\e)}}{|x|^{s_1}} + \frac{u^{p_2(\e)}}{|x|^{s_2}} = 0 \qquad \text{in } \Omega, \\
& u_{\e} =0, \qquad \text{on } \partial \Omega,
\end{aligned}
\end{eqnarray*}
and $\Phi_{\e} (u_{\e}) = c_{\e}$. Similar to  \eqref{22equation}, we have
$$
\int_{\Omega} |\nabla u_{\e}|^2 dx \le C_1,
$$
for some constant $C_1$ independent of $\e$.
By passing to a subsequence if necessary, we may assume 
$$
u_\e  \rightharpoonup  u  \qquad \text{in } \qquad H^1_0(\Omega).
$$
If $u \not \equiv$ 0, then clearly $u$ is a solution of \eqref{1equation} and Theorem~\ref{1theorem} is proved. So it remains to prove $u \not \equiv 0$ in $\Omega$.

Suppose $u \equiv 0$.  As in Section 3, there exists $x_{\e} \in \Omega$ such that 
$$
u_\e (x_\e) = \max u_{\e} (x) = m_\e \to + \infty,
$$
and after a linear transformation on $y$, we have 
$$
v_\e(y)  = \frac{u_\e (x_\e + k_\e y)}{m_\e} \to v(y)  \qquad \text{ in } \qquad C^2_{loc} (\rn),
$$
where $k_{\e} = m_\e^{-\frac{2}{N-2}  +\frac{\e}{ 2-s_1}}$ and $v$ satisfies
\begin{eqnarray*}
\begin{aligned}
& \Delta v + \lambda \frac{v^{p_1}}{|y|^{s_1}} + \frac{v^{p_2}}{|y|^{s_2}} =0 \qquad \text{in } \qquad \rn, \\
& v=0 \qquad \text{ on }  \qquad \rn.
\end{aligned}
\end{eqnarray*}
Also by a direct computation, we have 
\begin{equation}
\label{eqq4.5}
\liminf_{\e \to 0} \int_{\Omega} |\nabla u_{\e}|^2 dx \ge \int_{\rn} |\nabla v|^2 dy,
\end{equation}
and
\begin{equation}
\label{eqq4.6}
\lim_{\e \to 0} \int_{\Omega} \frac{u^{p_2(\e) +1}}{|x|^{s_2}} dx \ge \int_{\rn} \frac{v^{p_2 + 1}}{|y|^{s_2}} dy.
\end{equation}
By \eqref{22equation}, we have
$$
c_{\e}^*= (\frac{1}{2} - \frac{1}{p_1(\e) + 1}) \int_{\Omega} |\nabla u_\e|^2 dx
+\Big( \frac{1}{p_1(\e)+1} - \frac{1}{p_2(\e) + 1} \Big) \int_{\Omega}\frac{u_{\e}^{p_2(\e) +1}}{|x|^{s_2}} dx. 
$$
Thus, \eqref{eqq4.5} and \eqref{eqq4.6} yields 
\begin{eqnarray*}
%\begin{aligned}
c^* = \lim_{\e \to 0} c^*_{\e} = (\frac{1}{2} - \frac{1}{ 2^*(s_1)} ) \int_{\rn} |\nabla v|^2 dy +(\frac{1}{2^*(s_1)} - \frac{1}{2^*(s_2)}) \int_{\rn} \frac{v^{2^*(s_2)}}{|y^{s_2}|} dy = c_1,
%\end{aligned}
\end{eqnarray*} 
which contradicts to \eqref{eqq4.4}. Hence, $u \not \equiv 0$, and then Theorem~\ref{1theorem} is proved.

\section{The Case $s_2=0$}
\hspace{.6cm}
As discussed in Introduction, equation\,\eqref{1equation} with $\lambda<0$ has no least-energy solutions.
In this sectin, we will consider a perturbed equation from equation\,\eqref{1equation}:
\begin{equation}\label{41equation}
\begin{cases}
&\Delta u-\frac{u^{2^*(s)-1}}{|x|^s}+u^p+u^{\frac{N+2}{N-2}}=0\quad\text{in \,}\Omega,\\
&u(x)>0\quad\text{in \,}\Omega\text{ \,and \,}u(x)=0\quad\text{on \,}\partial\Omega.
\end{cases}
\end{equation}
\begin{thm}\label{41thm}
Suppose $2^*(s)-1<p<\frac{N+2}{N-2}$ and $N\geq 4$. Then equation\,\eqref{41equation} has a positive solution.
\end{thm}
\begin{proof}
Let
\begin{equation*}
\varPhi(u):=\frac{1}{2}\int_\Omega|\nabla u|^2dx+\frac{1}{2^*(s)}\int_\Omega\frac{(u^+)^{2^*(s)}}{|x|^s}dx
-\frac{1}{p+1}\int_\Omega (u^+)^{p+1}dx-\frac{N-2}{2N}\int_\Omega (u^+)^{\frac{2N}{N-2}}dx.
\end{equation*}
Choose $0\ne x_0\in\Omega$ and
\begin{equation*}
v_\mu(x)=\phi(x)\left(\frac{\mu}{1+\mu^2|x-x_0|^2}\right)^{\frac{N-2}{2}}
\end{equation*}
for large $\mu>0$, where $\phi(x)$ is a cut-off function near $x_0$. Then it is not difficult to show that
\begin{equation*}
\sup_{t\geq 0}\varPhi(t v_\mu)=\varPhi(t_0 v_\mu)<\frac{1}{N}S_N^{\frac{N}{2}}
\end{equation*}
provided that $2^*(s)-1<p<\frac{N+2}{N-2}$ and $\mu$ is sufficiently large. Let $u_0=t_0 v_\mu$, and
\begin{equation*}
c_*=\inf_{P\in{\cal P}}\max_{w\in P}\varPhi(w),
\end{equation*}
where ${\cal P}$ is the class of continuous paths in $H^1_0(\Omega)$ connecting $0$ and $u_0$. Then it is easy to see
\begin{equation}\label{42equation}
0<c_*<\frac{1}{N}S_N^{\frac{N}{2}}.
\end{equation}
We claim: $c_*$ is a critical value for $\varPhi$. By the deformation lemma\,(see lemma in \cite{}), there exists $u_j\in H^1_0(\Omega)$ such that
\begin{equation}\label{43equation}
\begin{cases}
&\varPhi(u_j)=c_*(1+o(1)),\\
&\int_\Omega |\nabla u_j|^2dx+\int_\Omega\frac{(u_j^+)^{2^*(s)}}{|x|^s}dx-\int_\Omega(u_j^+)^{p+1}dx-\int_\Omega (u_j^+)^{\frac{2N}{N-2}}dx=o(1)\|u_j\|_{H^1_0},
\end{cases}
\end{equation}
where $o(1)\to 0$ as $j\to+\infty$. By \eqref{42equation} and \eqref{43equation}, we have
\begin{equation*}
\int_\Omega |\nabla u_j|^2dx\leq C
\end{equation*}
for some constant $C$ independent of $j$. By passing to a subsequence if necessary, we assume
\begin{equation*}
\begin{cases}
&u_j\rightharpoonup u\quad\text{in \,}H^1_0(\Omega),\\
&u_j\to u\quad\text{in \,}L^{p+1}(\Omega).
\end{cases}
\end{equation*}
If $u\ne 0$, then it is easy to see $u$ is a solution of \eqref{41equation} and $\varPhi(u)=c_*$. Hence, it remains to show that $u\not\equiv 0$.
We prove it by contradiction.

\medskip

Now suppose $u\equiv 0$, and set
\begin{equation*}
A=\underline{\lim}_{j\to +\infty}\int_\Omega|\nabla u_j|^2dx,\quad B=\underline{\lim}_{j\to+\infty}\int_\Omega\frac{(u_j^+)^{2^*(s)}}{|x|^s}dx\text{ \,and \,}
C=\underline{\lim}_{j\to +\infty}\int_\Omega(u_j^+)^{\frac{2N}{N-2}}dx.
\end{equation*}
Then \eqref{43equation} implies
\begin{equation}\label{45equation}
c_*=\frac{A}{2}+\frac{B}{2^*(s)}-\frac{N-2}{2N}C,
\end{equation}
and
\begin{equation}\label{46equation}
C=A+B.
\end{equation}
By the Sobolev inequality, \eqref{46equation} implies
\begin{equation*}
C=A+B\geq A\geq S_N C^{1-\frac{2}{N}}.
\end{equation*}
Thus, we have
\begin{equation*}
C\geq S_N^{\frac{N}{2}}\text{ \,and \,}A\geq S_N C^{1-\frac{2}{N}}\geq S_N^{\frac{N}{2}}.
\end{equation*}
Then \eqref{45equation} yields
\begin{equation*}
c_*=(\frac{1}{2}-\frac{N-2}{2N})A+(\frac{1}{2^*(s)}-\frac{N-2}{2N})B\geq \frac{1}{N}S_N^{\frac{N}{2}},
\end{equation*}
a contradiction to \eqref{42equation}. Hence, Theorem\,\ref{41thm} is proved.
\end{proof}

\begin{align*}
& \text{YanYan Li}     &&  \qquad   &&&\text{Chang-Shou Lin} \\
&  \text{Department of Mathematics}                                  && \qquad &&&\text{Taida Institute for Mathematical Sciences} \\                
& \text{Rutgers University}   &&  \qquad  &&&\text{Department of Mathematics}  \\
&\text{Piscataway, NJ 08854, USA}   &&   \qquad &&&\text{National Taiwan University}   \\
& \text{yyli@math.rutgers.edu}   &&  \qquad &&&\text{Taipei 106, Taiwan}  \\
&   && \qquad  &&&\text{cslin@math.ntu.edu.tw}     \\  
\end{align*}

\end{document}